\documentclass[letterpaper,11pt,reqno]{amsart}
\usepackage{graphicx,amsmath,amssymb,amsthm, float}

\usepackage{srcltx}
\usepackage[latin1]{inputenc}
\usepackage[T1]{fontenc}

\usepackage[left=1in,right=1in, top=1in, bottom=1in]{geometry}

\newtheorem{X}{X}[section]

\newtheorem{lemma}[X]{Lemma}

\newtheorem{theorem}[X]{Theorem}
\newtheorem{conjecture}[X]{Conjecture}

\theoremstyle{definition}
\newtheorem{remark}[X]{Remark}

\newcommand{\sumstar}{\underset{d \text{\emph{ odd}}}{\sum \nolimits^{*}} w\left( \frac dX \right) }
\newcommand{\sumstarnoemph}{\underset{d \text{ odd}}{\sum \nolimits^{*}} w\left( \frac dX \right) }
\newcommand{\Fam}{\mathcal F^*(X)}
\newcommand{\RR}{\mathbb{R}}

\newcommand{\CC}{\mathbb{C}}

\newcommand{\eps}{\varepsilon}
\newcommand{\R}{\mathbb R}

\newcommand{\Dstar}{\mathcal D^*(\phi;X)}

\newcommand{\sumn}{\underset{d \text{ odd}}{\sum \nolimits^{*}} w\left( \frac dX \right) }
\newcommand{\sumtstar}{\underset{d \text{ odd}}{\sum \nolimits^{*}} w\left( \frac dX \right) }

\newcommand{\esupp}{\emph{\text{sup}}(\emph{\text{supp}}\hspace{1pt}\widehat \phi)}

\numberwithin{equation}{section}

\title[A transition in the Ratios Conjecture]{Low-lying zeros of quadratic Dirichlet $L$-functions: \\A transition in the Ratios Conjecture}

\author{Daniel Fiorilli, James Parks and Anders S\"odergren}
\address{D\'epartement de math\'ematiques et de statistique, Universit\'e d'Ottawa, \newline
	\rule[0ex]{0ex}{0ex}\hspace{8pt} 585 King Edward, Ottawa, Ontario, K1N 6N5, Canada}
\email{daniel.fiorilli@uottawa.ca}
\address{Institut f\"ur Algebra, Zahlentheorie und Diskrete Mathematik, \newline
	\rule[0ex]{0ex}{0ex}\hspace{8pt} Leibniz Universit\"at Hannover, Welfengarten 1, 30167 Hannover, Germany \newline \rule[0ex]{0ex}{0ex}\hspace{8pt} \textit{Present address}:   Department of Mathematics, KTH Royal Institute of Technology,\newline \rule[0ex]{0ex}{0ex}\hspace{8pt} Lindstedtsv\"agen, SE-100 44 Stockholm, Sweden}
\email{jparks@kth.se}
\address{Department of Mathematical Sciences, Chalmers University of Technology and the University \newline
	\rule[0ex]{0ex}{0ex}\hspace{8pt} of Gothenburg, SE-412 96 Gothenburg, Sweden}
\email{andesod@chalmers.se}

\date{\today}
\begin{document}

\begin{abstract}
We study the $1$-level density of low-lying zeros of quadratic Dirichlet $L$-functions by applying the $L$-functions Ratios Conjecture. We observe a transition in the main term as was predicted by the Katz-Sarnak heuristic as well as in the lower order terms when the support of the Fourier transform of the corresponding test function reaches the point $1$. Our results are consistent with those obtained in previous work under GRH and are furthermore analogous to results of Rudnick in the function field case.
\end{abstract}
	
\maketitle
	
\section{Introduction}
In this paper we study low-lying zeros in the family  
$$ \mathcal F^* (X):= \big\{ L(s,\chi_{8d}) : 1\leq |d| \leq X;  d \text{ is odd and squarefree} \big\}, $$
of Dirichlet $L$-functions. Here we define the real primitive character of conductor $8|d|$ by the Kronecker symbol 
$\chi_{8d}(n):=\left( \frac{8d}{n}\right).$ Low-lying zeros for quadratic Dirichlet $L$-functions have been studied extensively in the literature, see, for example, \"Ozl\"uk and Snyder \cite{OS2}, Rubinstein \cite{Rub}, Gao \cite{Gt}, Entin, Roditty-Gershon and Rudnick \cite{ERR}, Miller \cite{Mi}, and the authors' previous paper \cite{FPS2}.

The $L$-functions Ratios Conjecture of Conrey, Farmer and Zirnbauer \cite[Section 5]{CFZ} is a recipe for obtaining conjectural formulas for averages of quotients of (products of) $L$-functions evaluated at certain values in the critical strip. This can in turn be used to give extremely precise predictions for a variety of statistics for families of $L$-functions. Our goal will be to investigate low-lying zeros of quadratic Dirichlet $L$-functions using this tool. The present study complements the existing work done under the assumption of the Generalized Riemann Hypothesis (GRH), and gives further insights for a wide class of test functions that was previously out of reach. 
While this has already been carried out for $\mathcal F^* (X)$ and many other families (see, e.g., \cite{CS}), the novelty in our work is to isolate a sharp transition in the main and lower-order terms, which agrees with the prediction of Katz-Sarnak, Rudnick's work \cite{Rud} over function fields and our previous results \cite{FPS2}. 

Before we describe the new results, we first review and refine our previous work. We begin by introducing the
$1$-level density of the family  $\Fam$. Given a large positive number $X$, we set 
\begin{align*}
L:=\log\left(\frac{X}{2\pi e}\right)
\end{align*} 
and	\begin{align*}
W^*(X):=\sumtstar, 
\end{align*}
where $w(t)$ is an even, nonzero and nonnegative Schwartz function and the star on the sum denotes a restriction to squarefree integers. We introduce the $1$-level density of the family $\Fam$ as the functional
\begin{equation} \label{equation definition one level densities}
\Dstar:= \frac 1{W^*(X)}\sumtstar\sum_{\gamma_{8d}} \phi\left( \gamma_{8d}\frac{  L}{2\pi} \right).
\end{equation}
Here and throughout, $\phi$ will be a real and even Schwartz test function. Furthermore, we define $\gamma_{8d}:= -i(\rho_{8d}-\frac12)$, where $\rho_{8d}$ runs over the nontrivial zeros of $L(s,\chi_{8d})$ (i.e. zeros with $0<\Re(\rho_{8d}) <1$).
	
The Katz-Sarnak heuristic \cite{KS2} provides a precise prediction of the statistics of low-lying zeros in families of $L$-functions (see also the recent paper \cite{SST}). In our situation the heuristic asserts that $\mathcal F^*(X)$ has symplectic symmetry type and in particular that
\begin{equation}
\label{equation katz-sarnak prediciton}
\lim_{X \rightarrow \infty} \mathcal D^* (\phi;X) = \widehat\phi(0)-\frac 12 \int_{-1}^1 \widehat \phi(u)\,du 
\end{equation}
independently of the support of $\widehat \phi$. Moreover, note that there is a phase transition in the right-hand side of \eqref{equation katz-sarnak prediciton} occurring when the supremum $\sigma$ of the support of $\widehat \phi$ reaches $1$. This transition has a strong influence on the shape of the lower order terms in $\Dstar$ (cf. \cite{FPS2} and \cite{Rud}) and will thus be of fundamental interest in the current paper.
	
The lower order terms in the $1$-level density for the family $\Fam$ were recently studied under the assumption of GRH in \cite{FPS2}. There we obtained an asymptotic formula for $\Dstar$ in descending powers of $\log X$, which is valid when the support of $\widehat\phi$ is contained in $(-2,2)$. In particular, we uncovered a phase transition when $\sigma$ approaches $1$ in the main term as well as in the lower order terms. This asymptotic formula (see \cite[Theorem 3.5]{FPS2}) contains a term $J(X)$ which is expressed in terms of explicit transforms of the weight function $w$. In the current paper we give an asymptotic for this complicated expression (see Section \ref{section 2}), and as a consequence the following result holds.

\begin{theorem}\label{Theorem FPS2}
Fix $\eps>0$. Assume GRH and suppose that $\sigma=\esupp < 2$. Then the $1$-level density of low-lying zeros in the family $\Fam$ of quadratic Dirichlet $L$-functions whose conductor is an odd squarefree multiple of $8$ is given by
\begin{align}\label{Dstar FPS2}
\Dstar =\,\,& \widehat \phi(0)+\int_{1}^{\infty} \widehat \phi(u)\,du+ \frac{\widehat \phi(0)}L \bigg( \log (2 e^{1-\gamma}) + \frac 2{\widehat w(0)}\int_0^{\infty} w(x) (\log x)\,dx \bigg)   \nonumber\\
&+\frac{1}{L}\int_0^\infty\frac{e^{-x/2}+e^{-3x/2}}{1-e^{-2x}}\left(\widehat{\phi}(0)-\widehat{\phi}\left(\frac{x}{L}\right)\right) dx\\
&- \frac 2{L   }\sum_{\substack{p>2 \\ j\geq 1}} \frac{\log p}{p^{j}} \left(  1+\frac 1p\right)^{-1} \widehat \phi\left( \frac{2j \log p}{L} \right)  +J(X)+O_{\eps}\big(X^{\frac{\sigma}6-\frac13+\eps}\big), \notag
\end{align}
where $J(X)$ is defined in \eqref{equation J(X)} and satisfies the asymptotic relation\footnote{Here $\mathcal M w$ denotes the Mellin transform of $w$.}
\begin{equation}\label{equation J(X) FPS3} 
J(X)=\frac {\widehat \phi( 1)}L \bigg( -  \log\big(2^{\frac 73}e^{1+\gamma}\big) + 2\frac{\zeta'(2)}{\zeta(2)} -\frac{\mathcal M w'(1)}{\mathcal Mw(1)}\bigg)+O(L^{-2}).
\end{equation}
\end{theorem}

\begin{remark}
Note that Theorem \ref{Theorem FPS2} is in agreement with the Katz-Sarnak heuristic $\eqref{equation katz-sarnak prediciton}$ (cf. \cite[Section 3]{FPS2}). In fact, in \cite[Theorem 1.1]{FPS2} we show how to write $\eqref{Dstar FPS2}$ as $$\Dstar =  \widehat \phi(0)-\frac{1}2\int_{-1}^{1} \widehat \phi(u)\,du +\sum_{k=1}^K \frac {R_{w,k}(\phi)}{(\log X)^k} +O_{w,\phi,K}\left( \frac 1{(\log X)^{K+1}}\right), $$
where $R_{w,k}(\phi)$ are linear functionals in $\phi$ that can be given explicitly in terms of $w$ and the derivatives of $\widehat\phi$ at the points $0$ and $1$. An analogous expression having a transition at the point $1$ was obtained by Rudnick \cite[Corollary 3]{Rud} in the function field case.
\end{remark}

In Conjecture \ref{ratiosconjecture}, following the approach of \cite{CS}, we provide a Ratios Conjecture prediction for the family $\Fam$  with a power saving error term of size at most $O_{\varepsilon}(X^{-1/2+\varepsilon})$. 
In connection with our investigation of the low-lying zeros in $\Fam$, we use this conjecture to obtain the following asymptotic expression for $\Dstar$ (see Section \ref{The ratios conjecture's prediction}).
	
\begin{theorem}
\label{oneleveldensityresult}
Fix $\varepsilon>0$ and let $\phi$ be an even Schwartz test function on $\R$ whose Fourier transform has compact support. Assume GRH and Conjecture \ref{ratiosconjecture} (the Ratios Conjecture for $\Fam$). Then the $1$-level density for the low-lying zeros in the family $\Fam$ is given by
\begin{multline*}
\Dstar=\frac{1}{{W}^*(X)}\sumstar \frac{1}{2\pi}
\int_{\mathbb R} \bigg(2 \frac{\zeta'(1+2it)}{\zeta(1+2it)} + 2A_{\alpha}(it,it)+\log\left(\frac{8|d|}{\pi}\right) +\frac 12\frac{\Gamma'}{\Gamma}\left(\frac{1}{4}+\frac{\mathfrak{a}-it}{2}\right)\\
+\frac 12\frac{\Gamma'}{\Gamma}\left(\frac{1}{4}+\frac{\mathfrak{a}+it}{2}\right)-2 X_{d}\left(\tfrac{1}{2}+it\right)\zeta(1-2it)A(-it,it) \bigg) \,
\phi\left(\frac{tL}{2\pi}\right) \, dt+ O_{\varepsilon}\big(X^{-\frac 12+\varepsilon}\big),
\end{multline*}
where $A$, $A_{\alpha}$, $\mathfrak{a}$ and $X_d$ are defined by $\eqref{defnofae}$, $\eqref{AALPHAE}$, $\eqref{equation mathfrak a}$ and $\eqref{xe}$ respectively. 
\end{theorem}

Given the apparent difference between the formulas for $\Dstar$ in Theorem \ref{Theorem FPS2} and Theorem \ref{oneleveldensityresult}, it is an interesting question to ask whether the detailed information about the phase transition at $\sigma=1$ occurring in $\eqref{Dstar FPS2}$ is also present in Theorem \ref{oneleveldensityresult}. Our main theorem answers this question in the affirmative and to the best of our knowledge this is the first comprehensive investigation of such a transition via a Ratios Conjecture calculation (see Section \ref{section compare results}).
	
\begin{theorem}\label{Theorem ratios transition}
Let $\phi$ be an even Schwartz test function on $\R$ whose Fourier transform has compact support. Assume GRH and Conjecture \ref{ratiosconjecture} (the Ratios Conjecture for $\Fam$). Then the $1$-level density for the low-lying zeros in the family $\Fam$ can be written in the form
\begin{align*}
\Dstar =\,\,& \widehat \phi(0)+\int_{1}^{\infty} \widehat \phi(u)\,du+ \frac{\widehat \phi(0)}L \bigg( \log (2 e^{1-\gamma}) + \frac 2{\widehat w(0)}\int_0^{\infty} w(x) (\log x)\,dx \bigg)   \nonumber\\
&+\frac{1}{L}\int_0^\infty\frac{e^{-x/2}+e^{-3x/2}}{1-e^{-2x}}\left(\widehat{\phi}(0)-\widehat{\phi}\left(\frac{x}{L}\right)\right) dx\\
&- \frac 2{L   }\sum_{\substack{p>2 \\ j\geq 1}} \frac{\log p}{p^{j}} \left(  1+\frac 1p\right)^{-1} \widehat \phi\left( \frac{2j \log p}{L} \right)  \notag\\
&+\frac {\widehat \phi( 1)}L \bigg( -  \log\big(2^{\frac 73}e^{1+\gamma}\big) + 2\frac{\zeta'(2)}{\zeta(2)} -\frac{\mathcal M w'(1)}{\mathcal Mw(1)}\bigg)+O(L^{-2}).
\nonumber
\end{align*}
In particular, for test functions $\phi$ with $\sigma=\esupp <2$, this formula agrees with Theorem \ref{Theorem FPS2} up to an error term of order $O(L^{-2})$.
\end{theorem}
	 
\begin{remark}
The error term in Theorem \ref{Theorem ratios transition} results from a Taylor expansion of the integrand in \eqref{Big I} in powers of  $\tfrac \tau L$ with remainder $O(|\tau|^2L^{-2})$. In principle, with a higher order Taylor expansion one can obtain an error term in Theorem \ref{Theorem ratios transition} of size $O(L^{-k})$ for any $k\geq 2$ by the same methods. However, for conciseness, we choose not to pursue this here.
\end{remark}	 
	 
\begin{remark}
The first indications of a phase transition in a Ratios Conjecture calculation were given by Miller \cite[Section 2.2]{Mi}. However, the arguments in \cite{Mi} contain several serious issues. In particular, the Katz-Sarnak main term was not computed correctly (cf. \cite[Lemma 2.1, Lemma 2.6]{Mi}).
\end{remark}	

From the results in \cite{FPS2} and \cite{Rud} it is not clear whether or not we can expect the first lower order term to have the same shape also for test functions of larger support of their Fourier transforms. One of the consequences of Theorem \ref{Theorem ratios transition} is that under the assumption of the Ratios Conjecture no additional phase transitions occur when the support of the Fourier transform reaches points in the interval $[2,\infty)$.

\section{Proof of Theorem \ref{Theorem FPS2}}\label{section 2}

As in \cite{FPS2} we define	
\begin{equation}\label{equation J(X)}
J(X):= \frac 1L\int_{0}^{\infty} \bigg( \widehat \phi( 1+\tfrac {\tau}L  )   e^{\frac{\tau}2} \sum_{n \geq 1}  h_1\big(n e^{\frac{\tau}2}\big)  +\widehat \phi( 1-\tfrac {\tau}L  )  \sum_{n\geq 1} h_2\big( n e^{\frac{\tau}2}\big)\bigg)\,d\tau,
\end{equation}
where
\begin{align*} 
h_1(x) &:= \frac{3\zeta(2)}{\widehat w(0)}\sum_{\substack{ s\geq 1 \\ s \text{ odd}}} \frac {\mu(s)}s \big(\widehat g(2sx)-\widehat g(sx)\big),\\
h_2(x)&:=\frac{3\zeta(2)}{\widehat w(0)}\sum_{\substack{ s\geq 1 \\ s \text{ odd}}} \frac {\mu(s)}{s^2} \big( \tfrac 12g( \tfrac x{2s})-g(\tfrac xs)\big),  \end{align*}
and
\begin{equation*}
\label{equation definition g}
g(y):=\widehat w(4\pi e y^2).
\end{equation*}
The formula \eqref{Dstar FPS2} for $\Dstar$ was obtained in \cite[Theorem 3.5]{FPS2}. Thus it remains to prove the nontrivial estimate \eqref{equation J(X) FPS3} for $J(X)$. 
	
Using Mellin inversion, we get
\begin{align*}
h_1(x) &= \frac{3\zeta(2)}{\widehat w(0)}\frac 1{2 \pi i} \int_{(\frac32)} \sum_{\substack{ s\geq 1 \\ s \text{ odd}}} \frac {\mu(s)}{s^{1+z}}(2^{-z}-1) \mathcal M\widehat g(z)\frac{dz}{x^z}  \\
&= \frac{3\zeta(2)}{\widehat w(0)}\frac 1{2 \pi i} \int_{(\frac32)} \frac{(2^{-z}-1)}{( 1- 2^{-1-z})\zeta(1+z)} \mathcal M\widehat g(z)\frac{dz}{x^z}.
\end{align*}
Similarly,	
$$ h_2(x)=\frac{3\zeta(2)}{\widehat w(0)}\frac 1{2 \pi i} \int_{(-\frac12)} \frac{(2^{-z-1}-1)}{( 1- 2^{-2-z})\zeta(2+z)} \mathcal M g(-z)x^zdz,$$
and shifting the contour of integration to the left we obtain  
$$ h_2(x)=\frac{3\zeta(2)}{\widehat w(0)}\frac 1{2 \pi i} \int_{(-\frac54)} \frac{(2^{-z-1}-1)}{( 1- 2^{-2-z})\zeta(2+z)} \mathcal M g(-z)x^zdz.$$
Hence we write \eqref{equation J(X)} as
\begin{align}\label{equation preTaylor}
 J(X)=& \frac{3\zeta(2)}{L\widehat w(0)} \frac 1{2 \pi i}\int_{0}^{\infty} \bigg( \widehat \phi( 1+\tfrac {\tau}L  )   \int_{(\frac32)} \frac{(2^{-z}-1)\zeta(z)}{( 1- 2^{-1-z})\zeta(1+z)} \mathcal M\widehat g(z)\frac{dz}{e^{(z-1)\tau/2}} \notag\\ &+\widehat \phi( 1-\tfrac {\tau}L  )   \int_{(-\frac54)} \frac{(2^{-z-1}-1)\zeta(-z)}{( 1- 2^{-2-z})\zeta(2+z)} \mathcal M g(-z)e^{z\tau/2}dz \bigg)\,d\tau \notag\\
 =& \frac{3\zeta(2)}{L\widehat w(0)}\frac 1{2 \pi i}\bigg( \int_{(\frac32)} \frac{(2^{-z}-1)\zeta(z)}{( 1- 2^{-1-z})\zeta(1+z)} \mathcal M\widehat g(z) \int_{0}^{\infty}  \widehat \phi( 1+\tfrac {\tau}L  )e^{-(z-1)\tau/2}d\tau dz\notag\\
 &+   \int_{(-\frac54)} \frac{(2^{-z-1}-1)\zeta(-z)}{( 1- 2^{-2-z})\zeta(2+z)} \mathcal M g(-z) \int_{0}^{\infty}\widehat\phi( 1-\tfrac {\tau}L  ) e^{z\tau/2}d\tau dz \bigg)\,.
\end{align}	
Now we use the Taylor expansions of $\widehat\phi( 1+\tfrac {\tau}L  )$ and $\widehat\phi( 1-\tfrac {\tau}L  )$ in \eqref{equation preTaylor}. From the constant terms in these expansions we obtain
\begin{align*}
&\frac{3\zeta(2)}{L\widehat w(0)}\frac{ 2\widehat \phi( 1)}{2 \pi i}\bigg( \int_{(\frac32)} \frac{(2^{-z}-1)\zeta(z)}{( 1- 2^{-1-z})\zeta(1+z)} \mathcal M\widehat g(z) \frac{ dz}{z-1}-   \int_{(-\frac54)} \frac{(2^{-z-1}-1)\zeta(-z)}{( 1- 2^{-2-z})\zeta(2+z)} \mathcal M g(-z) \frac{dz}z \bigg)\,	\\
&= \frac{3\zeta(2)}{L\widehat w(0)}\frac{ 2\widehat \phi( 1)}{2 \pi i}\bigg( \int_{(\frac12)} \frac{(2^{-z-1}-1)\zeta(z+1)}{( 1- 2^{-2-z})\zeta(2+z)} \mathcal M\widehat g(z+1) \frac{ dz}{z}-  \int_{(-\frac54)} \frac{(2^{-z-1}-1)\zeta(-z)}{( 1- 2^{-2-z})\zeta(2+z)} \mathcal M g(-z) \frac{dz}z \bigg)\\
&=  \frac{3\zeta(2)}{L\widehat w(0)}\frac{ 2\widehat \phi( 1)}{2 \pi i}\bigg( \int_{(\frac 12)} \frac{(2^{-z-1}-1)\zeta(z+1)}{( 1- 2^{-2-z})\zeta(2+z)} \mathcal M\widehat g(z+1) \frac{ dz}{z}-   \int_{(\frac 12)} \frac{(2^{-z-1}-1)\zeta(-z)}{( 1- 2^{-2-z})\zeta(2+z)} \mathcal M g(-z) \frac{dz}z \bigg)\\
&\hspace{10pt}+\frac {\widehat \phi( 1)}L \bigg( \log( 2^{\frac 23} \pi^2)+2\frac{\zeta'(2)}{\zeta(2)}-\frac 2{\widehat w(0)} \int_0^{\infty} (\log x) g'(x)dx \bigg)\\
&=  \frac{3\zeta(2)}{L\widehat w(0)}\frac{ 2\widehat \phi( 1)}{2 \pi i} \int_{(\frac 12)} \frac{(2^{-z-1}-1)}{( 1- 2^{-2-z})\zeta(2+z)}\Big(\zeta(z+1)\mathcal M\widehat g(z+1)-\zeta(-z) \mathcal M g(-z)    \Big)\frac{dz}z \\
&\hspace{10pt}+\frac {\widehat \phi( 1)}L \bigg( \log( 2^{\frac 23} \pi^2)+2\frac{\zeta'(2)}{\zeta(2)}-\frac 2{\widehat w(0)} \int_0^{\infty} (\log x )g'(x)dx \bigg).
\end{align*}
	
Next we note that
\begin{equation}\label{equation zeta mellin}
\zeta(z+1)\mathcal M\widehat g(z+1)=\zeta(-z) \mathcal M g(-z). 
\end{equation} 
Indeed, applying Plancherel's identity, we have that
\begin{align*}
\mathcal M\widehat g(z+1) &= \int_0^{\infty} x^z \widehat g(x)dx= \frac 12\int_{\mathbb R} |x|^z \widehat g(x)dx \\
&= -\frac{\sin(\pi z/2) \Gamma(z+1)}{(2\pi)^{z+1} } \int_{\mathbb R}  \frac{ g(x)}{|x|^{z+1}}dx= -\frac{2\sin(\pi z/2) \Gamma(z+1)}{(2\pi)^{z+1} } \mathcal Mg(-z),
\end{align*}
where the Fourier transform of $|x|^z$ is in the sense of distributions (cf. \cite[Exercise 8.7]{F}). Hence \eqref{equation zeta mellin} follows by an application of the functional equation of the zeta function. We conclude that 
$$ J(X)= \frac {\widehat \phi( 1)}L \bigg( \log( 2^{\frac 23} \pi^2)+2\frac{\zeta'(2)}{\zeta(2)}-\frac 2{\widehat w(0)} \int_0^{\infty} (\log x) g'(x)dx \bigg)+O(L^{-2})$$
(cf.\ the proof of \cite[Lemma 3.6]{FPS2}).
	
Now we consider the integral in the above formula. For small $\eta>0$ we have that
\begin{align*}
\int_0^{\infty} (\log x) g'(x)dx &=\int_0^{\infty}8\pi e x (\log x) \widehat w'(4\pi e x^2)dx = \frac 12\int_0^{\infty} \log  \bigg(\frac x{4\pi e}\bigg) \widehat w'(x)dx \\
&= \frac{\widehat w(0)}{2} \log (4\pi e)  +\frac 12\int_{\eta}^{\infty} (\log x) \widehat w'(x)dx +O\big(\eta \log (\eta^{-1})\big) \\
&=\frac{\widehat w(0)}{2} \log (4\pi e)+\frac 12 \Big[ (\log x) \widehat w(x) \Big]_{\eta}^{\infty} - \frac 12\int_{\eta}^{\infty} \frac{\widehat w(x)}xdx+ O\big(\eta \log (\eta^{-1})\big) \\ 
&=\frac{\widehat w(0)}{2} \log (4\pi e)+ \frac{\widehat w(0)}2 \log(\eta^{-1})  - \frac 12\int_{\eta}^{\infty} \frac{\widehat w(x)}xdx +O\big(\eta \log (\eta^{-1})\big) \\
&= \frac{\widehat w(0)}{2} \log (4\pi e) - \frac 12\int_{\eta}^{\infty} \frac{\widehat w(x)-\widehat w(0)I_{[0,1]}(x)}xdx +O\big(\eta \log (\eta^{-1})\big).
\end{align*}
Hence, by taking the limit as $\eta$ tends to zero and using \cite[Example (e) on page 132]{V}, we obtain
\begin{align*}
\int_0^{\infty} (\log x )g'(x)dx&= \frac{\widehat w(0)}{2} \log (4\pi e) -\frac 14\int_{\mathbb R} \frac{\widehat w(x)-\widehat w(0)I_{[-1,1]}(x)}{|x|}dx\\
&= \frac{\widehat w(0)}{2} \log (4\pi e) +\frac 12\int_{\mathbb R} \big(\gamma+\log |2\pi x|\big)w(x) dx\\
&= \frac{\widehat w(0)}{2} \log \big(2^3\pi^2 e^{1+\gamma}\big) +\int_{0}^{\infty} (\log x)w(x) dx. 
\end{align*}
We conclude that
\begin{align*}
 J(X)=& \frac {\widehat \phi( 1)}L \bigg( \log\big( 2^{\frac 23} \pi^2\big)+2\frac{\zeta'(2)}{\zeta(2)}-\frac 2{\widehat w(0)} \bigg(  \frac{\widehat w(0)}{2} \log \big(2^3\pi^2 e^{1+\gamma}\big) +\int_{0}^{\infty} (\log x)w(x) dx\bigg)\bigg) +O(L^{-2}) \\
 =&\frac {\widehat \phi( 1)}L \bigg( -  \log\big(2^{\frac 73}e^{1+\gamma}\big) + 2\frac{\zeta'(2)}{\zeta(2)} -\frac{\mathcal M w'(1)}{\mathcal Mw(1)}\bigg)+O(L^{-2}).
\end{align*}	
This verifies the identity \eqref{equation J(X) FPS3} and thus completes the proof of Theorem \ref{Theorem FPS2}.

\section{The Ratios Conjecture's prediction}\label{The ratios conjecture's prediction}
	
\subsection{The Ratios Conjecture}
	
In this section we formulate an appropriate version of the Ratios Conjecture. Such a calculation was already performed by Conrey and Snaith \cite{CS} in the family of $L$-functions associated to even real Dirichlet characters. For our purposes we need to derive an analogous conjecture with an additional smooth weight function for the family $\mathcal F^* (X)$.
	
To begin, we consider the sum
\begin{equation} \label{ralphgam}
R(\alpha,\gamma):=\frac1{W^*(X)}\sumn \frac{L\left(\frac{1}{2}+\alpha,\chi_{8d}\right)}{L\left(\frac{1}{2}+\gamma,\chi_{8d}\right)}.
\end{equation}
For $L(s,\chi_{8d})$ we have 
\begin{equation} \label{lemobius}
\frac{1}{L(s,\chi_{8d})} = \sum_{n=1}^\infty
\frac{\mu(n)\chi_{8d}(n)}{n^s}
\end{equation}
and the approximate functional equation
\begin{equation}
L\left(s, \chi_{8d}\right) = \sum_{n<x} \frac{\chi_{8d}(n)}{n^s}+\left(\frac{\pi}{8|d|}\right)^{s-\frac{1}{2}}
\frac{\Gamma\left(\frac{1+\mathfrak{a}-s}{2}\right)}{\Gamma\left(\frac{\mathfrak{a}+s}{2}\right)} \sum_{n<y}
\frac{\chi_{8d}(n)}{n^{1-s}}+\:{\rm Error},\label{approxeq}
\end{equation}
where $xy=4|d|/\pi$ and
\begin{equation}\label{equation mathfrak a}
\mathfrak{a}:=\mathfrak{a}(d)=\begin{cases} 0 & \text{ if }d\geq 0, \\ 1 & \text{ if } d<0.\end{cases}
\end{equation}
We now follow the Ratios Conjecture recipe \cite{CFZ} (see also the presentation in \cite{CS}) and disregard the error term and complete the sums (i.e.\ replace $x$ and $y$ with infinity).
	
The first step is to replace the numerator of $\eqref{ralphgam}$ with the approximate functional equation $\eqref{approxeq}$ (with the above modification) and the denominator of $\eqref{ralphgam}$ with $\eqref{lemobius}$. We first focus on the principal sum from $\eqref{approxeq}$  evaluated at $s=\frac{1}{2}+\alpha$, which gives the contribution
\begin{equation} \label{ragone}
R_1(\alpha, \gamma):=\frac1{W^*(X)}\sumn \sum_{h,m} \frac{\mu(h)\chi_{8d}(hm)}{h^{\frac{1}{2}+\gamma}m^{\frac{1}{2} +\alpha}}
\end{equation}
to \eqref{ralphgam}. The next step in the Ratios Conjecture procedure is to replace $\chi_{8d}(hm)$ in $\eqref{ragone}$ with its weighted average over the family $\Fam$. From \cite[Lemma 2.2 and Remark 2.3]{FPS2} we have that
\begin{align}\label{average coeff}
\frac{1}{W^*(X)}\sumn\chi_{8d}(hm)&=\begin{cases}\prod_{p\mid hm}\big(\frac{p}{p+1}\big)+O_{h,m,\varepsilon}\big(X^{-\frac34+\varepsilon}\big) & {\rm if} \: hm = \text{ odd }\square,\\
O_{h,m,\varepsilon}\big(X^{-\frac34+\varepsilon}\big)  & {\rm otherwise}.
\end{cases}
\end{align} 
Thus the main contribution to the sum in $\eqref{ragone}$ occurs when $hm$ is an odd square and following the recipe we disregard the non-square terms and the error terms. Hence $\eqref{ragone}$ is replaced with
\begin{equation*}
\widetilde R_1(\alpha, \gamma):=
\sum_{hm = \text{odd } \square} \frac{\mu(h)}{h^{\frac{1}{2}+\gamma}m^{\frac{1}{2} + \alpha}}\prod_{p\mid hm}\frac{p}{p+1}
\end{equation*}
and writing this as an Euler product gives
\begin{align*}
\widetilde R_1(\alpha, \gamma)&=\prod_{p>2}\Bigg(1+\frac{p}{p+1}\sum_{\substack{k,\ell \geq 0 \\ k+\ell>0 \\ k+\ell \:{\rm even}}}\frac{\mu(p^k)}{p^{k\left(\frac{1}{2}+\gamma\right)+\ell(\frac{1}{2}+\alpha)}}\Bigg)\\
&=\prod_{p>2}\Bigg(1+\frac{p}{p+1}\bigg(\sum_{\ell=1}^{\infty}\frac{1}{p^{\ell(1+2\alpha)}}-\frac1{p^{1+\alpha+\gamma}}\sum_{\ell=0}^{\infty}\frac{1}{p^{\ell(1+2\alpha)}}\bigg)\Bigg)\\
&=\prod_{p>2}\Bigg(1+\frac{p}{(p+1)(1-p^{-1-2\alpha})}\bigg(\frac1{p^{1+2\alpha}}-\frac1{p^{1+\alpha+\gamma}}\bigg)\Bigg).
\end{align*}
In the above Euler product we factor out zeta functions corresponding to the divergent parts of $\widetilde R_1(\alpha,\gamma)$ (as $\alpha,\gamma\to0$). This results in 
\begin{align}\label{ragtilde1}
\widetilde R_1(\alpha,\gamma)= \frac{\zeta(1+2\alpha)}{\zeta(1+\alpha+\gamma)}A(\alpha,\gamma),
\end{align}
where
\begin{equation}
A(\alpha,\gamma):=\left(\frac{2^{1+\alpha+\gamma}-2^{\gamma-\alpha}}{2^{1+\alpha+\gamma}-1}\right)\prod_{p>2}\left(1-\frac{1}{p^{1+\alpha+\gamma}}\right)^{-1}\left(1-\frac{1}{(p+1)p^{1+2\alpha}}-\frac{1}{(p+1)p^{\alpha+\gamma}}\right).\label{defnofae}
\end{equation}
Note that the product $A(\alpha,\gamma)$ is absolutely convergent for $\Re(\alpha),\Re(\gamma)>-\frac{1}4$. 
	
Next we consider the contribution of the dual sum coming from the approximate functional equation (the second sum in $\eqref{approxeq}$) to $\eqref{ralphgam}$, namely the sum
\begin{equation} \label{ragtwo}
R_2(\alpha, \gamma):=\frac1{W^*(X)}\sumn X_{d}\left(\tfrac{1}{2} + \alpha\right)\sum_{h,m} \frac{\mu(h)\chi_{8d}(hm)}
{h^{\frac{1}{2}+\gamma}m^{\frac{1}{2}-\alpha}},
\end{equation}
where
\begin{equation} \label{xe}
X_{d}(s):=\frac{\Gamma\left(\frac{1+\mathfrak{a}-s}{2}\right)}{\Gamma\left(\frac{\mathfrak{a}+s}{2}\right)}\left(\frac{\pi}{8|d|}\right)^{s-\frac{1}{2}}.
\end{equation}
As above we follow the Ratios Conjecture procedure and replace $\chi_{8d}(hm)$ in $\eqref{ragtwo}$ with its weighted average over the family $\Fam$. Using \eqref{average coeff}, we replace \eqref{ragtwo} with
\begin{equation*}
\widetilde R_2(\alpha, \gamma):=\frac{1}{W^*(X)}\sumn X_{d}\left(\tfrac{1}{2} + \alpha\right)\widetilde R_1(-\alpha,\gamma).
\end{equation*}
Finally, using the formula \eqref{ragtilde1} we state the Ratios Conjecture for our weighted family of quadratic Dirichlet $L$-functions as:
	
\begin{conjecture} \label{ratiosconjecture}
Let $\varepsilon>0$ and let $w$ be an even and nonnegative Schwartz test function on $\R$ which is not identically zero. Assume GRH and suppose that the complex numbers $\alpha$ and $\gamma$ satisfy $|\Re(\alpha)|< \frac{1}{4}$, $\frac{1}{\log X} \ll\Re(\gamma)<\frac{1}{4}$ and $\Im(\alpha),\Im(\gamma) \ll X^{1-\varepsilon}$. Then we have that
\begin{multline*}
\frac{1}{W^*(X)}\sumstar\frac{L(\frac{1}{2} +\alpha,\chi_{8d})}{L(\frac{1}{2} + \gamma,\chi_{8d})}= 
\frac{\zeta(1+2\alpha)}{\zeta(1+\alpha+\gamma)}A(\alpha,\gamma)\\
+\frac{1}{W^*(X)}\sumstar X_{d}\left(\tfrac{1}{2}+\alpha\right) \frac{\zeta(1-2\alpha)}{\zeta(1-\alpha+\gamma)}A(-\alpha,\gamma)+ O_\varepsilon\big(X^{-\frac12+\varepsilon}\big),
\end{multline*}
where $A(\alpha,\gamma)$ is defined in $\eqref{defnofae}$ and $X_d(s)$ is defined in $\eqref{xe}$.\footnote{The error term $O_{\varepsilon}(X^{-\frac12+\varepsilon})$ is part of the statement of the Ratios Conjecture. Note also that the extra conditions on $\alpha$ and $\gamma$, which were not used in the derivation of Conjecture \ref{ratiosconjecture}, are included here as standard conditions under which conjectures produced by the Ratios Conjecture recipe are expected to hold.}
\end{conjecture}
	
In our calculation of the $1$-level density (see Section \ref{1levelsection}), we require the average of the logarithmic derivative of the $L$-functions in $\Fam$. We set
\begin{equation}\label{AALPHAE}
A_{\alpha}(r,r) := \frac{\partial}{\partial \alpha}A(\alpha,\gamma)\bigg|_{\alpha=\gamma=r}.
\end{equation}
	
\begin{lemma} \label{ratiostheorem} 
Let $\varepsilon>0$ and let $w$ be an even and nonnegative Schwartz test function on $\R$ which is not identically zero. Suppose that $r \in \CC$ satisfies $\frac{1}{\log X} \ll\Re(r)<\frac{1}{4}$ and $\Im(r)\ll X^{1-\varepsilon}$. Then, assuming GRH and Conjecture \ref{ratiosconjecture}, we have that
\begin{multline}
\frac{1}{W^*(X)}\sumstar \frac{L'(\frac{1}{2} + r,\chi_{8d})}{L(\frac{1}{2} + r,\chi_{8d})}= \frac{\zeta'(1+2r)}{\zeta(1+2r)}+A_{\alpha}(r,r)\\
-\frac{1}{W^*(X)}\sumstar X_{d}\left(\tfrac{1}{2}+r\right) \zeta(1-2r)A(-r,r)+ O_{\varepsilon}\big(X^{-\frac12+\varepsilon}\big).\label{ratiostheoremthree}
\end{multline}
\end{lemma}
	
\begin{proof}
Observing that $A(r,r)=1$, we get 
\begin{equation*}
\frac{\partial}{\partial \alpha}\frac{\zeta(1+2\alpha)}{\zeta(1+\alpha+\gamma)}A(\alpha,
\gamma)\bigg|_{\alpha=\gamma=r} = \frac{\zeta'(1+2r)}{\zeta(1+2r)}+A_{\alpha}(r,r)
\end{equation*}
and 
\begin{multline*}
\frac{\partial}{\partial \alpha}\frac{1}{W^*(X)}\sumn X_{d}\left(\tfrac{1}{2}+\alpha\right)\frac{\zeta(1-2\alpha)}{\zeta(1-\alpha+\gamma)}A(-\alpha,\gamma)\bigg|_{\alpha=\gamma=r}\\
=-\frac{1}{W^*(X)}\sumn X_{d}\left(\tfrac{1}{2}+r\right)\zeta(1-2r)A(-r,r),
\end{multline*}
which gives the main term in \eqref{ratiostheoremthree}. Finally, a straightforward argument using Cauchy's integral formula for derivatives shows that the error term remains the same under differentiation. 
\end{proof}

\subsection{The Ratios Conjecture's prediction for the $1$-level density}\label{1levelsection}
	
In this section we derive a first formulation of the Ratios Conjecture's prediction for the $1$-level density of low-lying zeros in the family $\Fam$.

\begin{proof}[Proof of Theorem \ref{oneleveldensityresult}]
We recall that
\begin{equation*}
\Dstar=\frac{1}{W^*(X)}\sumn \sum_{\gamma_{8d}}\phi\left(\gamma_{8d}\frac{L}{2\pi}\right).
\end{equation*}
Using the argument principle, we obtain
\begin{equation} \label{cauchydens}
\Dstar=\frac{1}{W^*(X)}\sumn \frac{1}{2\pi i}\left(\int_{(c)} - \int_{(1-c)}\right)\frac{L'(s,\chi_{8d})}{L(s,\chi_{8d})}
\phi\left(\frac{-iL}{2\pi}\left(s-\frac{1}{2}\right)\right) ds
\end{equation}
with $\frac{1}{2}+\frac{1}{\log X}<c<\frac34$. For the integral in \eqref{cauchydens} on the line with real part $1-c$, we make the change of variables $s \mapsto 1-s$. Recalling that $\phi$ is even, we find that this integral equals
\begin{align}
\frac{1}{W^*(X)}\sumn \frac{1}{2\pi i} \int_{(c)} \frac{L'(1-s,\chi_{8d})}{L(1-s,\chi_{8d})}
\phi\left(\frac{-iL}{2\pi}\left(s-\frac{1}{2}\right)\right)ds.
\label{loneminuss}
\end{align}
Next, applying the functional equation
\begin{align*}
\Lambda(s,\chi_{8d}):=\left(\frac{8|d|}{\pi}\right)^{\frac{1}{2}(s+\mathfrak{a})}\Gamma\left(\frac{s+\mathfrak{a}}{2}\right)L(s,\chi_{8d})=\Lambda(1-s,\chi_{8d})
\end{align*}
(cf.\ \cite[Sect.\ 9]{D}), together with $\eqref{xe}$, we obtain
\begin{equation}\label{labellprimeoverl}\textbf{}
\frac{L'(s,\chi_{8d})}{L(s,\chi_{8d})} = \frac{X_{d}'(s)}{X_{d}(s)} -
\frac{L'(1-s,\chi_{8d})}{L(1-s,\chi_{8d})}\,.
\end{equation}
Using $\eqref{loneminuss}$ and $\eqref{labellprimeoverl}$ together with the change of variables $s = \tfrac 12 + r$, we obtain
\begin{align}
\Dstar
=& \frac{1}{W^*(X)}\sumn \frac{1}{2\pi i}\int_{(c-\tfrac 12)} \bigg(2 \frac{L'\left(\tfrac12 + r,\chi_{8d}\right)}{L\left(\tfrac12 + r,\chi_{8d}\right)} - \frac{X_{d}'\left(\tfrac12+r\right)}{X_{d}\left(\tfrac12+r\right)}\bigg) \phi\left(\frac{iLr}{2\pi}\right)dr.
\label{checknine}
\end{align}
		
Changing the order of summation and integration in \eqref{checknine}, we substitute
\begin{equation*}
\frac{1}{W^*(X)}\sumn\frac{L'\left(\tfrac12+r,\chi_{8d}\right)}{L\left(\tfrac12 + r,\chi_{8d}\right)}
\end{equation*}
with the right-hand side of $\eqref{ratiostheoremthree}$. Note that this substitution is valid only when $\Im(r)<X^{1-\varepsilon}$. However, since $\widehat \phi$ has compact support  on $\R$ the function $\phi\left(\frac{iLr}{2\pi}\right)$ is rapidly decaying as $|\Im(r)|\to\infty$. From this fact and the estimate \cite[Thm.\ 5.17]{IK} of the logarithmic derivative of Dirichlet $L$-functions, we can bound the tail of the integral in \eqref{checknine} by 	$O_{\varepsilon}\left(X^{-1+\varepsilon}\right)$. Furthermore, using a similar argument to bound the tail of the integral in \eqref{RATIOSONELEV}, we obtain
\begin{multline}
\Dstar=\frac{1}{W^*(X)}\sumn \frac{1}{2\pi i}\int_{(c-\tfrac 12)}  \bigg(2 \frac{\zeta'(1+2r)}{\zeta(1+2r)}+2A_{\alpha}(r,r)-\frac{X_{d}'\left(\tfrac12+r\right)}{X_{d}\left(\tfrac12+r\right)} \\
-2X_{d}\left(\tfrac{1}{2}+r\right)\zeta(1-2r)A(-r,r)\bigg) \phi\left(\frac{iLr}{2\pi}\right)dr + O_{\varepsilon}\big(X^{-\frac 12+\varepsilon}\big).\label{RATIOSONELEV}
\end{multline}
		
The final step is to move the contour of integration from $\Re(r)=c-\tfrac 12=c'$ to
$\Re(r)=0$. Note that the function $$2 \frac{\zeta'(1+2r)}{\zeta(1+2r)}+2A_{\alpha}(r,r)-\frac{X_{d}'\left(\tfrac12+r\right)}{X_{d}\left(\tfrac12+r\right)}-2X_{d}\left(\tfrac{1}{2}+r\right)\zeta(1-2r)A(-r,r)$$
is analytic in this region. Thus, by Cauchy's Theorem, we have that
\begin{align*} \nonumber
\Dstar=&\frac{1}{W^*(X)}\sumn \frac{1}{2\pi}
\int_{\mathbb R} \bigg(2 \frac{\zeta'(1+2it)}{\zeta(1+2it)} + 2A_{\alpha}(it,it)+\log\left(\frac{8|d|}{\pi}\right) +\frac 12\frac{\Gamma'}{\Gamma}\left(\frac{1}{4}+\frac{\mathfrak{a}-it}{2}\right)\\
+&\frac 12\frac{\Gamma'}{\Gamma}\left(\frac{1}{4}+\frac{\mathfrak{a}+it}{2}\right)-2 X_{d}\left(\tfrac{1}{2}+it\right)\zeta(1-2it)A(-it,it) \bigg) \,
\phi\left(\frac{tL}{2\pi}\right) \, dt+ O_{\varepsilon}\big(X^{-\frac 12+\varepsilon}\big),
\end{align*}
since $$\frac{X_{d}'(\tfrac 12+it)}{X_{d}(\tfrac 12+it)}=\log\left(\frac{\pi}{8|d|}\right) -\frac12\frac{\Gamma'}{\Gamma}\left(\frac{1}{4}+\frac{\mathfrak{a}-it}{2}\right)-\frac12\frac{\Gamma'}{\Gamma}\left(\frac{1}{4}+\frac{\mathfrak{a}+it}{2}\right).$$ 
This completes the proof.
\end{proof}

\section{Making the Ratios Conjecture's prediction explicit}\label{section compare results}

In this section we wish to compare the Ratios Conjecture's prediction for $\Dstar$ in Theorem~\ref{oneleveldensityresult} with the result in Theorem \ref{Theorem FPS2}. We recall the prediction obtained in \eqref{RATIOSONELEV}, which for  $\frac 1{\log X}<c'< \frac 14$ equals
\begin{multline}
\Dstar=\frac{1}{W^*(X)}\sumn \frac{1}{2\pi i}\int_{(c')}  \bigg(2 \frac{\zeta'(1+2r)}{\zeta(1+2r)}+2A_{\alpha}(r,r)+\log\left(\frac{8|d|}{\pi}\right) +\frac 12\frac{\Gamma'}{\Gamma}\left(\frac{1}{4}+\frac{\mathfrak{a}-r}{2}\right)\\
+\frac 12\frac{\Gamma'}{\Gamma}\left(\frac{1}{4}+\frac{\mathfrak{a}+r}{2}\right) 
-2X_{d}\left(\tfrac{1}{2}+r\right)\zeta(1-2r)A(-r,r)\bigg) \phi\left(\frac{iLr}{2\pi}\right)dr + O_{\varepsilon}\big(X^{-\frac 12+\varepsilon}\big).\label{Dstar ratios}
\end{multline}
To begin, we focus on the first two terms in $\eqref{Dstar ratios}$. Recalling $\eqref{defnofae}$, we write
$$A(\alpha,\gamma)=\left(1+\frac{2^{1+\alpha}-2^{1+\gamma}}{3\cdot 2^{1+2\alpha+\gamma}-2^\gamma-2^{1+\alpha}}\right)\prod_{p}\left(1-\frac{1}{p^{1+\alpha+\gamma}}\right)^{-1}\left(1-\frac{1}{(p+1)p^{1+2\alpha}}-\frac{1}{(p+1)p^{\alpha+\gamma}}\right).$$
Using the fact that $A(r,r)=1$, we compute 
$$A_\alpha(r,r)=\frac{2\log 2}{3(2^{1+2r}-1)}+\sum_p\frac{\log p}{(p+1)(p^{1+2r}-1)}.$$
Furthermore, we have that
\begin{equation*}
\frac{\zeta'(1+2r)}{\zeta(1+2r)}=-\sum_{p} \frac{\log p}{p^{1+2r}-1}.
\end{equation*} 
We now have the following result. 
	
\begin{lemma} \label{lemma: prime sum}
Let $\eps>0$ and suppose that $\tfrac1{\log X}<c'<\tfrac 14$. Then
\begin{equation}
\frac{1}{2\pi i}\int_{(c')}  \bigg(2 \frac{\zeta'(1+2r)}{\zeta(1+2r)}+2A_{\alpha}(r,r)\bigg) \phi\left(\frac{iLr}{2\pi}\right)dr	=- \frac 2{L   }\sum_{\substack{p>2 \\ j\geq 1}} \frac{\log p}{p^{j}} \left(  1+\frac 1p\right)^{-1} \widehat \phi\left( \frac{2j \log p}{L} \right).
\label{equation lemma 4.1}
\end{equation}
\end{lemma}

\begin{remark}
In \cite[Lemma 3.7]{FPS2}, it was shown that the right-hand side of \eqref{equation lemma 4.1} is asymptotic to $-\phi(0)/2$ as $X\rightarrow \infty$. The Katz-Sarnak prediction \eqref{equation katz-sarnak prediciton} is obtained by combining this term with the main terms occuring in Lemmas \ref{lemma logd} and \ref{lemma I evaluation}.
\end{remark}

\begin{proof}[Proof of Lemma \ref{lemma: prime sum}]
We have that
\begin{align}
\frac{1}{2\pi i}\int_{(c')}&  \bigg(2 \frac{\zeta'(1+2r)}{\zeta(1+2r)}+2A_{\alpha}(r,r) \bigg) \phi\left(\frac{iLr}{2\pi}\right)dr\nonumber\\
=&\frac{1}{\pi i}\int_{(c')}\bigg(-\sum_{p}\frac{\log p}{p^{1+2r}-1}+\frac{2\log 2}{3(2^{1+2r}-1)}+\sum_p\frac{\log p}{(p+1)(p^{1+2r}-1)}\bigg)\phi\left(\frac{ iLr}{2 \pi}\right)dr\nonumber\\
=&\frac{1}{\pi i}\int_{(c')}\bigg(\frac{2\log 2}3\sum_{j=1}^{\infty}\frac{1}{2^{(1+2r)j}}-\sum_p \log p\sum_{j=1}^{\infty}\frac{1}{p^{(1+2r)j}}+\sum_p\frac{\log p}{(p+1)}\sum_{j=1}^{\infty}\frac{1}{p^{(1+2r)j}}\bigg)\phi\left(\frac{ iLr}{2 \pi}\right)dr\nonumber\\
=&\frac{1}{\pi i}\int_{(c')}\bigg(\frac{2\log 2}3\sum_{j=1}^{\infty}\frac{1}{2^{(1+2r)j}} -\sum_p \frac{p\log p}{p+1}\sum_{j=1}^{\infty}\frac{1}{p^{(1+2r)j}}\bigg)\phi\left(\frac{ iLr}{2 \pi}\right)dr.
\label{ANOTHERONELEVELDENSITY}
\end{align}
Making the substitution $u=-\frac{iLr}{2\pi}$, we have that  $\eqref{ANOTHERONELEVELDENSITY}$ becomes
\begin{equation}
\frac{2}{L}\int_{\mathcal C'}
\Bigg(\frac{2\log 2}{3} \sum_{j=1}^{\infty} \frac{e^{-(\log 2)\left(\frac{4\pi i uj}{L}\right)}}{2^{j}}-\sum_p \frac{p\log p}{p+1} \sum_{j=1}^{\infty} \frac{e^{-(\log p)\left(\frac{4\pi i uj}{L}\right)}}{p^{j}}\Bigg)\phi\left(u\right) du,\label{theseventwo}
\end{equation}
where $\mathcal C'$ denotes the horizontal line $\Im(u)=-\frac{Lc'}{2\pi}$. We note that the summations inside the integral over $\mathcal C'$ in $\eqref{theseventwo}$ converge absolutely and uniformly on compact subsets. Thus we may interchange the order of integration and summation and $\eqref{theseventwo}$ becomes
\begin{align*}
\frac{4\log 2}{3L} \sum_{j=1}^{\infty} \frac{1}{2^{j}}\int_{\mathcal C'}\phi\left(u\right)e^{-2\pi i u\left(\frac{2j\log 2}{L}\right)}du-\frac{2}{L}\sum_p \frac{p\log p}{p+1} \sum_{j=1}^{\infty} \frac{1}{p^{j}}\int_{\mathcal C'}\phi\left(u\right)e^{-2\pi i u\left(\frac{2j\log p}{L}\right)}du.
\end{align*}
Next, we move the contour of integration from $\mathcal C'$ to the line $\Im(u) =0$. Note that this is allowed since $\widehat \phi$ has compact support on $\R$ and the entire function $\phi(z):=\int_{\R}\widehat\phi(x)e^{2\pi ixz}\,dx$ satisfies the inequality
\begin{align*}
\left|\phi(T+it)\right|\leq\frac1{2\pi|T|}\int_{\R}\big|\widehat\phi'(x)\big|\max(1,e^{xLc'})\,dx,
\end{align*}
uniformly for $-\frac{Lc'}{2\pi}\leq t\leq0$, as $T\to\pm\infty$. Thus \eqref{ANOTHERONELEVELDENSITY} equals
\begin{align*}
-\frac{2}{L}\sum_{p>2} \frac{p\log p}{p+1} \sum_{j=1}^{\infty} \frac{1}{p^{j}}\widehat \phi\left(\frac{2j\log p}{L}\right),
\end{align*}
which concludes the proof.
\end{proof}

We now study the third, fourth and fifth terms in \eqref{Dstar ratios}. For these terms we can shift the line of integration to the imaginary axis, since the integrands are analytic in this region. We now give estimates for these shifted integrals.

\begin{lemma} \label{lemma logd}
Fix $\varepsilon>0$. We have
\begin{multline*} 
\frac{1}{W^*(X)}\sumstar \frac{1}{2\pi}
\int_{\mathbb R} \log\left(\frac{8|d|}{\pi}\right)\phi\left(\frac{tL}{2\pi}\right) \, dt\\
=\widehat \phi(0)+ \frac{\widehat \phi(0)}{L}\bigg(\log\left(2^4 e \right)+ \frac 2{\widehat w(0)}\int_0^{\infty} w(x) (\log x) \,dx\bigg)   +O_{\varepsilon,w}\big(X^{-\frac 12 +\varepsilon}\big).  
\end{multline*}
\end{lemma}
	
\begin{proof}
The result follows as in \cite[Lemma 2.4]{FPS2} (see also \cite[Lemma 2.8]{FPS}).
\end{proof}
	
\begin{lemma} \label{lemma gammas}
We have
\begin{multline*} 
\frac{1}{{W}^*(X)}\sumstar \frac{1}{4\pi}
\int_{\mathbb R} \left(\frac{\Gamma'}{\Gamma}\left(\frac{1}{4}+\frac{\mathfrak{a}-it}{2}\right)
+\frac{\Gamma'}{\Gamma}\left(\frac{1}{4}+\frac{\mathfrak{a}+it}{2}\right)\right)\phi\left(\frac{tL}{2\pi}\right) \, dt\\
= \frac{\widehat{\phi}(0)}{L}\log\big(2^{-3}e^{-\gamma}\big)+ \frac{1}{L}\int_0^\infty\frac{e^{-x/2}+e^{-3x/2}}{1-e^{-2x}}\left(\widehat{\phi}(0)-\widehat{\phi}\left(\frac{x}{L}\right)\right) dx.
\end{multline*}
\end{lemma}
	
\begin{proof}
We have that
\begin{align*}
\frac{1}{{W}^*(X)}&\sumstarnoemph \frac{1}{4\pi}
\int_{\mathbb R} \left(\frac{\Gamma'}{\Gamma}\left(\frac{1}{4}+\frac{\mathfrak{a}-it}{2}\right)
+\frac{\Gamma'}{\Gamma}\left(\frac{1}{4}+\frac{\mathfrak{a}+it}{2}\right)\right)\phi\left(\frac{tL}{2\pi}\right) \, dt\\
=&\frac{1}{4L}\int_{\RR}\left(\frac{\Gamma'}{\Gamma}\left(\frac{1}{4}+\frac{\pi i u}{L}\right)+\frac{\Gamma'}{\Gamma}\left(\frac{1}{4}-\frac{\pi i u}{L}\right)+\frac{\Gamma'}{\Gamma}\left(\frac{3}{4}+\frac{\pi i u}{L}\right)+\frac{\Gamma'}{\Gamma}\left(\frac{3}{4}-\frac{\pi i u}{L}\right) \right)\phi(u)\, du\\
=&\left(\frac{\Gamma'}{\Gamma}\left(\frac 14\right) + \frac{\Gamma'}{\Gamma}\left(\frac 34\right)\right)\frac{\widehat{\phi}(0)}{2L} +\frac{1}{L}\int_0^\infty\frac{e^{-x/2}+e^{-3x/2}}{1-e^{-2x}}\left(\widehat{\phi}(0)-\widehat{\phi}\left(\frac{x}{L}\right)\right) dx,
\end{align*}
from  \cite[Lemma 12.14]{MV}. The result follows.
\end{proof}
	
The next lemma is required to evaluate the last term in \eqref{Dstar ratios}.
	
\begin{lemma}\label{lemma: d pretentious :p}
Let $\eps>0$ and assume that $0\leq \Re(r) \leq \frac 12$. Then  we have the estimate
\begin{equation*}
\frac{1}{W^*(X)}\sumstar |d|^{-r} = \frac{2}{\widehat{w}(0)}X^{-r}\mathcal M w(1-r) +O_{\eps,w}\left(\big(|\Im(r)|+1\big)^{\frac {13}{84}+\eps}X^{-\frac12-\Re (r)+\eps}\right).
\end{equation*}
\end{lemma}
	
\begin{proof}
We write the sum we are interested in as the Mellin integral
$$ S:=\sumn |d|^{-r}= \frac 1{2\pi i} \int_{(2)}\frac{2^{s+r+1}}{2^{s+r}+1} \frac{\zeta(s+r)}{\zeta(2(s+r))} X^s \mathcal M w(s) ds. $$
We pull the contour of integration to the line $\Re(s)=\tfrac 12 -\Re(r)+\eps.$ Then we have a contribution from the simple pole at $s=1-r$. Note that the restriction on $\Re(r)$ ensures that we do not encounter the potential pole of $\mathcal{M}w(s)$ at $s=0$. Hence, by the rapid decay of $\mathcal M w$ on vertical lines, Bourgain's subconvexity bound \cite[Theorem 5]{B} on $|\zeta(s+r)|$ and the boundedness of $|\zeta(2(s+r))^{-1}|$, we have that
$$ S= \frac{4}{3\zeta(2)}X^{1-r}\mathcal M w(1-r)  +O_{\eps,w}\left(\big(|\Im(r)|+1\big)^{\frac {13}{84}+\eps}X^{\frac12-\Re(r)+\eps}\right). $$
Similarly as in \cite[Lemma 2.10]{FPS}, it can be shown that $$ W^*(X)=\frac{2}{3\zeta(2)}X\widehat w(0) +O_{\eps,w}\big(X^{\frac12 +\eps}\big).$$ We complete the proof by combining the last two estimates.
\end{proof}

We now give an asymptotic formula for the last term in \eqref{Dstar ratios}. By a straightforward computation and by recalling the definitions of $X_d$ and $A$, we apply the substitution $r=2\pi i \tau / L$ to obtain
\begin{multline}\label{Big I}
I:=\frac{-\zeta(2)}{L}\int_{\mathcal C'}
\Bigg(\frac{\Gamma\left(\tfrac14 -\tfrac{\pi i \tau}L\right)}{\Gamma\left(\tfrac14 +\tfrac{\pi i \tau}L\right)}+\frac{\Gamma\left(\tfrac34 -\tfrac{\pi i \tau}L\right)}{\Gamma\left(\tfrac34 +\tfrac{\pi i \tau}L\right)} \Bigg)\left(\frac{\pi}{8}\right)^{\frac{2 \pi i \tau}{L}}\left(1+\frac{2-2^{\frac{4\pi i \tau}L+1}}{4-2^{\frac{4\pi i \tau}L}}\right)\frac{\zeta\left(1-\tfrac{4\pi i \tau}L\right)}{\zeta\left(2-\tfrac{4\pi i \tau}L\right)} \phi\left(\tau\right)\\
\times\frac{1}{W^*(X)}\sumn |d|^{-\frac{2\pi i \tau}L} d\tau,
\end{multline}
where $\mathcal C'$ again denotes the horizontal line $\Im(\tau)=-\frac{Lc'}{2\pi}$.
		
\begin{lemma}\label{lemma I evaluation} 
We have the asymptotic formula 
$$ I = \frac{\phi(0)}2 -\frac 12 \int_{-1}^1 \widehat \phi(\tau)d\tau+\frac {\widehat \phi(1)}L \bigg(  -\gamma -1 - \log\big(2^{\frac 73}\big) + 2\frac{\zeta'(2)}{\zeta(2)} -\frac{\mathcal M w'(1)}{\mathcal Mw(1)} \bigg) +O(L^{-2}). $$
\end{lemma}	
	
\begin{proof}
Let $\eta>0$ be small. We first change the contour of integration in $I$ to the path
$$ C:=C_0\cup C_1\cup C_2, $$ 
where
$$C_0:= \{ \Im(\tau)= 0 , |\Re(\tau)| \geq L^\eps \},
\quad C_1:= \{ \Im(\tau) = 0 ,  \eta \leq |\Re(\tau)| \leq L^\eps \}, \quad C_2:= \{ |\tau| = \eta, \Im(\tau) \leq 0 \}.  $$	
For the part of the integral $I$ over $C_0$, we trivially bound the sum over $d$ and use the rapid decay of $\phi$ to obtain the bound
\begin{align*}
&\ll_{\eps} \frac 1L \int_{C_0} \max\big(\log(|\tau|+2),L|\tau|^{-1}\big) (|\tau|+2)^{-3/\eps^2}d\tau \\
&\ll  \int_{L^{\eps}}^{L} \log(|\tau|+2) |\tau|^{-1} (|\tau|+2)^{-3/\eps^2}d\tau +L^{-2/\eps^2} \\
&\ll L^{-3/\eps} (\log L)^2 + L^{-2/\eps^2} \ll L^{-2/\eps}. 
\end{align*}  
On $C_1\cup C_2$, we use Taylor expansions for each factor in the integrand of $I$, except the last in which we apply Lemma \ref{lemma: d pretentious :p}. This yields
\begin{align*}
&-\frac{\zeta(2)}L\Bigg(\frac{\Gamma\left(\tfrac14 -\tfrac{\pi i \tau}L\right)}{\Gamma\left(\tfrac14 +\tfrac{\pi i \tau}L\right)}+\frac{\Gamma\left(\tfrac34 -\tfrac{\pi i \tau}L\right)}{\Gamma\left(\tfrac34 +\tfrac{\pi i \tau}L\right)} \Bigg)\left(\frac{\pi}{8}\right)^{\frac{2 \pi i \tau}{L}}\left(1+\frac{2-2^{\frac{4\pi i \tau}L+1}}{4-2^{\frac{4\pi i \tau}L}}\right)\frac{\zeta\left(1-\tfrac{4\pi i \tau}L\right)}{\zeta\left(2-\tfrac{4\pi i \tau}L\right)}\frac{ \phi\left(\tau\right)}{W^*(X)}\sumn |d|^{-\frac{2\pi i \tau}L} \\
&=-\frac 1L \bigg( 2+ (\gamma+\log 8)\frac{4\pi i \tau}L +O\bigg( \frac{|\tau|^2}{L^2}\bigg)\bigg)\bigg( 1+\frac{2\pi i \tau }{L} \log\Big(\frac{\pi}8\Big)  +O\bigg( \frac{|\tau|^2}{L^2}\bigg)\bigg) \\
&\quad \times\bigg( 1+ \bigg( 2\frac{\zeta'(2)}{\zeta(2)} - \frac 43 \log 2 \bigg) \frac{2\pi i \tau}L  +O\bigg( \frac{|\tau|^2}{L^2}\bigg)\bigg) \bigg( -\frac L{4\pi i \tau}+\gamma  +O\bigg( \frac{|\tau|}{L}\bigg)\bigg) \phi(\tau) \\ 
& \quad \times\frac{2}{\widehat{w}(0)} \bigg(X^{-\frac{2\pi i \tau}L}\mathcal M w\Big(1-\frac{2\pi i \tau}L\Big)  +O_{\eps,w}\bigg(\bigg(\frac{|\Im(i\tau)|}L+1\bigg)^{\frac{13}{84}+\eps}X^{-\frac12+\eps}\bigg)\bigg) \\
&= \frac 1{2\pi i\tau} \bigg( 1+ \frac{2\pi i \tau }L\bigg( -\gamma-\log\bigg(\frac{2^{\frac43}}{\pi}\bigg)+2\frac{\zeta'(2)}{\zeta(2)}\bigg) +O\bigg( \frac{|\tau|^2}{L^2}\bigg)\bigg)\phi(\tau) e^{-2\pi i \tau} \\
& \quad \times\bigg(1-\frac{2\pi i \tau \log (2\pi e)}{L}  +O\bigg( \frac{|\tau|^2}{L^2}\bigg) \bigg)\bigg(1-\frac{\mathcal M w'(1)}{\mathcal Mw(1)} \frac{2\pi i \tau}L +O\bigg( \frac{|\tau|^2}{L^2}\bigg) \bigg) +O_{\eps,w}\big(X^{-\frac12+\eps}\big) \\ 
&=  \frac 1{2\pi i\tau} \bigg( 1+ \frac{2\pi i \tau }L\bigg( -\gamma -1- \log\big(2^{\frac 73}\big) + 2\frac{\zeta'(2)}{\zeta(2)} -\frac{\mathcal M w'(1)}{\mathcal Mw(1)} \bigg) +O\bigg( \frac{|\tau|^2}{L^2}\bigg)\bigg)\phi(\tau) e^{-2\pi i \tau}+O_{\eps,w}\big(X^{-\frac12+\eps}\big).
\end{align*}
Hence 
$$ I = I_1+I_2+ O_{w}(L^{-2}),$$
where 
$$ I_1 := \int_{C_1\cup C_2} \frac 1{2\pi i\tau}\phi(\tau) e^{-2\pi i \tau} d\tau$$ and	
$$ I_2 := \frac 1L \bigg( -\gamma -1- \log\big(2^{\frac 73}\big) + 2\frac{\zeta'(2)}{\zeta(2)} -\frac{\mathcal M w'(1)}{\mathcal Mw(1)} \bigg) \int_{C_1\cup C_2}\phi(\tau) e^{-2\pi i \tau} d\tau.$$ 

For the second integral $I_2$, by the rapid decay of $\phi$ on the real line and the holomorphy of the integrand, we have that
\begin{align*}
I_2& = \frac 1L \bigg( -\gamma -1- \log\big(2^{\frac 73}\big) + 2\frac{\zeta'(2)}{\zeta(2)} -\frac{\mathcal M w'(1)}{\mathcal Mw(1)} \bigg) \int_{\mathbb R}\phi(\tau) e^{-2\pi i \tau} d\tau+O_{\eps}(L^{-2/\eps}) \\
&=  \frac {\widehat \phi(1)}L \bigg( -\gamma-1 - \log\big(2^{\frac 73}\big) + 2\frac{\zeta'(2)}{\zeta(2)} -\frac{\mathcal M w'(1)}{\mathcal Mw(1)} \bigg) +O_{\eps}(L^{-2/\eps}).
\end{align*}	
Similarly, for the first integral $I_1$, we find that
$$ I_1 = \frac 1{2\pi i}\int_{C_0\cup C_1\cup C_2}e^{-2\pi i \tau}  \frac {\phi(\tau)}{\tau} d\tau +O_{\eps}(L^{-2/\eps})=J_1+J_2 +O_{\eps}(L^{-2/\eps}),$$
where 
$$ J_1:=\frac 1{2\pi i}\int_{C_0\cup C_1\cup C_2} \cos(2\pi \tau) \frac {\phi(\tau)}{\tau} d\tau $$
and
$$ J_2:=-\frac 1{2\pi }\int_{C_0\cup C_1\cup C_2} \sin(2\pi \tau) \frac {\phi(\tau)}{\tau} d\tau. $$
We see that since the integrand in $J_1$ is odd, the part of the integral on $C_0\cup C_1$ is zero. Hence
$$  J_1 =\frac 1{2\pi i}\int_{C_2} \cos(2\pi \tau) \frac {\phi(\tau)}{\tau} d\tau, $$
which by the residue theorem tends to $\phi(0)/2$ as $\eta$ tends to zero. As for the integral $J_2$, we apply Plancherel's identity. Since $\sin(2\pi \tau)/\tau$ is an entire function,
\begin{align*}
J_2=-\int_{\mathbb R} \frac{\sin(2\pi \tau)}{2\pi \tau}\phi(\tau)d\tau= -\frac 12 \int_{\mathbb R} I_{[-1,1]}(\tau)\widehat \phi(\tau)d\tau = -\frac 12 \int_{-1}^1 \widehat \phi(\tau)d\tau,
\end{align*}
which coincides with the second term in the Katz-Sarnak prediction. Since all of our error terms are independent of $\eta$, we conclude the desired result.
\end{proof}

\begin{proof}[Proof of Theorem \ref{Theorem ratios transition}]
The proof follows by combining Lemmas \ref{lemma: prime sum},  \ref{lemma logd}, \ref{lemma gammas} and \ref{lemma I evaluation}.
\end{proof}

\section*{Acknowledgments}
We thank Bruno Martin for helpful discussions. The first author was supported by an NSERC discovery grant. The third author was supported by a grant from the Swedish Research Council (grant 2016-03759).

\end{document}